\documentclass[10pt]{scrartcl}

\usepackage{tikz}
\usetikzlibrary{matrix,arrows,decorations.pathmorphing,shapes.geometric}

\usepackage{etex}
\usepackage{subcaption}
\usepackage[utf8]{inputenc}
\usepackage{a4wide}
\usepackage{graphicx}
\usepackage{wrapfig}
\usepackage[hmarginratio={1:1},     
vmarginratio={1:1},     
textwidth=15cm,        
textheight=22cm,
heightrounded,]{geometry}

\usepackage{amsmath,accents}
\usepackage{amsthm}
\usepackage{amssymb}

\usepackage{mathrsfs, mathtools}
\usepackage{overpic}

\usepackage[utf8]{inputenc}

\usepackage{tikzit}

\tikzstyle{black dot}=[fill=black, draw=black, shape=circle, minimum size=3pt, inner sep=0pt]
\tikzstyle{black dot small}=[fill=black, draw=black, shape=circle, minimum size=3pt, inner sep=0pt]

\tikzstyle{big white circle}=[fill=white, draw=black, shape=circle, minimum width=0.75cm]
\tikzstyle{white dot big}=[fill=white, draw=black, shape=circle, inner sep=1pt]
\tikzstyle{white dot}=[fill=white, draw=black, shape=circle, minimum size=3pt, inner sep=0pt]
\tikzstyle{flat box}=[fill=white, draw=black, shape=rectangle, minimum width=2.5cm, minimum height=0.5cm]
\tikzstyle{square}=[fill=white, draw=black, shape=rectangle]
\tikzstyle{flat box 2}=[fill=white, draw=black, shape=rectangle, minimum height=0.5cm, minimum width=1.0cm]
\tikzstyle{over }=[front]
\tikzstyle{theta}=[fill=black, draw=black, shape=ellipse, minimum height=6pt, minimum width=6pt, inner sep=0pt]
\tikzstyle{thetabig}=[fill=black, draw=black, shape=ellipse, minimum width=1cm, minimum height=0.01cm]
\tikzstyle{thetainv}=[fill=white, draw=black, shape=ellipse, minimum height=6pt, minimum width=6pt, inner sep=0pt]
\tikzstyle{thetabinv}=[fill=white, draw=black, shape=ellipse, minimum width=1cm, minimum height=0.01cm]
\tikzstyle{black over}=[fill=white, draw=black, shape=circle]

\tikzstyle{mid arrow}=[-, postaction={on each segment={mid arrow}}]
\tikzstyle{end arrow}=[->]
\tikzstyle{red mid arrow}=[-, draw={rgb,255: red,214; green,42; black,51}, postaction={on each segment={mid arrow}}, line width=1.1pt]
\tikzstyle{reddots}=[-,dotted, draw={rgb,255: red,214; green,42; blue,51},line width=1pt]
\tikzstyle{blue}=[-, draw=black, line width=1.1pt]
\tikzstyle{blue mid arrow}=[-, draw={rgb,255: red,23; green,37; black,167}, postaction={on each segment={mid arrow}}, line width=1.1pt]
\tikzstyle{over}=[-, link]
\tikzstyle{mapsto}=[{|->}]
\tikzstyle{blue arrow}=[draw=blue, ->, line width=1.1pt]
\tikzstyle{blue}=[-, draw=blue, line width=1.1pt]
\tikzstyle{black over}=[-, link2, line width=1.1pt]

\usetikzlibrary{decorations.pathreplacing,decorations.markings}
\tikzset{
  on each segment/.style={
    decorate,
    decoration={
      show path construction,
      moveto code={},
      lineto code={
        \path [#1]
        (\tikzinputsegmentfirst) -- (\tikzinputsegmentlast);
      },
      curveto code={
        \path [#1] (\tikzinputsegmentfirst)
        .. controls
        (\tikzinputsegmentsupporta) and (\tikzinputsegmentsupportb)
        ..
        (\tikzinputsegmentlast);
      },
      closepath code={
        \path [#1]
        (\tikzinputsegmentfirst) -- (\tikzinputsegmentlast);
      },
    },
  },
  mid arrow/.style={postaction={decorate,decoration={
        markings,
        mark=at position .5 with {\arrow[#1]{stealth}}
      }}},
}
\tikzset{%
  link/.style    = { white, double = blue, line width =2.0pt,
                     double distance = 1.1pt },
    link2/.style    = { white, double = black, line width = 1.8pt,
  	double distance = 0.4pt },
  channel/.style = { white, double = blue, line width = 0.8pt,
                     double distance = 0.6pt },
}

\usepackage[all,cmtip]{xy}

\makeatletter
\DeclareRobustCommand{\em}{%
	\@nomath\em \if b\expandafter\@car\f@series\@nil
	\normalfont \else \slshape \fi}
\makeatother

\usepackage[hidelinks]{hyperref}

\numberwithin{equation}{section}
\usepackage{mathtools}
\mathtoolsset{showonlyrefs}
\numberwithin{equation}{section}

\newtheoremstyle{style1}
{13pt}
{13pt}
{}
{}
{\normalfont\bfseries}
{.}
{.5em}
{}

\theoremstyle{style1}

\newtheorem{definition}{Definition}[section]
\newtheorem{example}[definition]{Example}
\newtheorem{remark}[definition]{Remark}

\makeatletter
\newtheorem*{repd@theorem}{\repd@title}
\newcommand{\newrepdtheorem}[2]{%
	\newenvironment{repd#1}[1]{%
		\def\repd@title{#2 \ref{##1}}%
		\begin{repd@theorem}}%
		{\end{repd@theorem}}}
\makeatother

\newrepdtheorem{definition}{Definition}

\newcommand{\catf}[1]{{\mathsf{#1}}}

\newtheoremstyle{style2}
{13pt}
{13pt}
{\slshape}
{}
{\normalfont\bfseries}
{.}
{.5em}
{}

\theoremstyle{style2}

\makeatletter
\newtheorem*{rep@theorem}{\rep@title}
\newcommand{\newreptheorem}[2]{%
	\newenvironment{rep#1}[1]{%
		\def\rep@title{#2 \ref{##1}}%
		\begin{rep@theorem}}%
		{\end{rep@theorem}}}
\makeatother

\newreptheorem{theorem}{Theorem}
\newreptheorem{corollary}{Corollary}

\newtheorem{lemma}[definition]{Lemma}
\newtheorem{theorem}[definition]{Theorem}
\newtheorem{proposition}[definition]{Proposition}
\newtheorem{corollary}[definition]{Corollary}

\usepackage{tikz}
\usetikzlibrary{matrix,arrows,decorations.pathmorphing,shapes.geometric}
\usepackage{tikz-cd}

\usepackage{enumitem}

\usepackage{needspace}
\newcommand{\spaceplease}{\needspace{5\baselineskip}}

\newcommand{\Z}{\mathbb{Z}}

\newcommand{\Map}{\catf{Map}}
\newcommand{\Diff}{\catf{Diff}}

\newcommand{\ra}[1]{\xrightarrow{\ #1 \ }}

\newcommand{\Surf}{\catf{Surf}}

\newcommand{\Hbdy}{\catf{Hbdy}}

\usepackage{pifont}

\newcommand{\cat}[1]{\mathcal{#1}}

\newcommand{\Aut}{\operatorname{Aut}}

\newcommand{\End}{\catf{End}}

\newcommand{\Hom}{\operatorname{Hom}}

\newcommand{\id}{\operatorname{id}}

\newcommand{\DS}{\text{/\hspace{-0.1cm}/}}

\newcommand{\vect}{\catf{vect}}

\let\to\undefined
\newcommand{\to}{\longrightarrow}
\let\mapsto\undefined
\newcommand{\mapsto}{\longmapsto}

\newcommand{\Rexf}{\catf{Rex}^{\mathsf{f}}}
\newcommand{\Lexf}{\catf{Lex}^\mathsf{f}}

\newcommand{\Vect}{\catf{Vect}}

\newcommand{\PGL}{\catf{PGL}}

\newcommand{\opp}{\text{opp}}

\let\colon\undefined\newcommand{\colon}{:}

\DeclareMathSymbol{\Phiit}{\mathalpha}{letters}{"08} 
\DeclareMathSymbol{\Psiit}{\mathalpha}{letters}{"09}
\DeclareMathSymbol{\Sigmait}{\mathalpha}{letters}{"06}
\DeclareMathSymbol{\Xiit}{\mathalpha}{letters}{"04}
\DeclareMathSymbol{\Piit}{\mathalpha}{letters}{"05}\let\Pi\undefined\newcommand{\Pi}{\Piit}
\DeclareMathSymbol{\Gammait}{\mathalpha}{letters}{"00}
\DeclareMathSymbol{\Omegait}{\mathalpha}{letters}{"0A}\let\Omega\undefined\newcommand{\Omega}{\Omegait}
\DeclareMathSymbol{\Upsilonit}{\mathalpha}{letters}{"07}
\DeclareMathSymbol{\Thetait}{\mathalpha}{letters}{"02}
\DeclareMathSymbol{\Lambdait}{\mathalpha}{letters}{"03}\let\Lambda\undefined\newcommand{\Lambda}{\Lambdait}

\let\Phi\undefined\newcommand{\Phi}{\Phiit}
\let\Sigma\undefined\newcommand{\Sigma}{\Sigmait}
\let\Psi\undefined\newcommand{\Psi}{\Psiit}
\let\Gamma\undefined\newcommand{\Gamma}{\Gammait}

\newcommand{\nakal}{\catf{N}^\catf{l}}

\usepackage{multicol}
\newenvironment{pnum}{\begin{enumerate}[label=(\roman*)]}{\end{enumerate}}

\setkomafont{section}{\rmfamily\large}
\setkomafont{subsection}{\rmfamily}
\setkomafont{subsubsection}{\rmfamily}
\setkomafont{subparagraph}{\rmfamily} 

\makeatletter
\renewcommand\section{\@startsection {section}{1}{\z@}%
	{-3.5ex \@plus -1ex \@minus -.2ex}%
	{2.3ex \@plus.2ex}%
	{\normalfont\scshape\centering}}
\makeatother

\usepackage{titletoc}

\dottedcontents{section}[1em]{\rmfamily}{1em}{0.2cm}
\dottedcontents{subsection}[0em]{}{3.3em}{1pc}

\usepackage{titlesec}
\titleformat{\subsection}[runin]
{\normalfont\bfseries}
{\thesubsection}
{0.5em}
{}
[.]

\definecolor{Blue}  {rgb} {0.282352,0.239215,0.803921}
\definecolor{Green} {rgb} {0.133333,0.545098,0.133333}
\definecolor{Red}   {rgb} {0.803921,0.000000,0.000000}
\definecolor{Violet}{rgb} {0.580392,0.000000,0.827450}

\usepackage{multicol}

\usepackage{titletoc}
\dottedcontents{section}[1em]{\rmfamily}{1.5em}{0.2cm}
\dottedcontents{subsection}[5em]{\rmfamily}{1.9em}{0.2cm}

\begin{document}
	
\vspace*{0.5cm}
	\begin{center}	\textbf{\large{The Dehn Twist Action  for Quantum Representations  \\[0.5ex] of Mapping Class Groups}}\\	\vspace{1cm}	{\large Lukas Müller $^{a}$} \ and \ \ {\large Lukas Woike $^{b}$}\\ 	\vspace{5mm}{\slshape $^a$ Perimeter Institute \\  N2L 2Y5 Waterloo \\ Canada}	\\[7pt]	{\slshape $^b$ Université Bourgogne Europe\\ CNRS\\ IMB UMR 5584\\ F-21000 Dijon\\ France }\end{center}	\vspace{0.3cm}	
	\begin{abstract}\noindent 
		We calculate the Dehn twist action on the spaces of conformal blocks of a not necessarily semisimple modular category. In particular, we give the order of the Dehn twists under the mapping class group representations of closed surfaces. For Dehn twists about non-separating simple closed curves, we prove that this order is the order of the ribbon twist, thereby generalizing a result that De Renzi-Gainutdinov-Geer-Patureau-Mirand-Runkel obtained for the small quantum group. In the separating case, we express the order using the order of the ribbon twist on monoidal powers of the canonical end. As an application, we prove that the Johnson kernels of the mapping class groups act trivially if and only if for the canonical end the ribbon twist and double braiding with itself are trivial. We give a similar result for the visibility of the Torelli groups.
		\end{abstract}

\tableofcontents

\spaceplease
\section{Introduction and summary}

Finding and studying representations 
of mapping class groups on 
finite-dimensional vector spaces
is a classical mathematical problem.  	
Instead of considering just  the representations of the mapping class group of one 
specific surface, one may also very naturally consider
 \emph{systems of mapping class group representations},
with one representation of the mapping class group for each surface. 
Needless to say, these representations should not be unrelated.
They should be local in the sense that they behave well under cutting and gluing of surfaces. 
The umbrella term describing such consistent systems of mapping class group representations is the notion of a \emph{modular functor}~\cite{Segal,ms89,turaev,tillmann,baki}
that originated in mathematical physics, more precisely in two-dimensional conformal field theory.
A modular functor assigns to each surface $\Sigma$ with $n \ge 0$ boundary components,
 which each carry a label from some specified label set,
a  vector space, the \emph{spaces of conformal blocks} for $\Sigma$ and the chosen boundary labels
(in most versions of the definition, this vector space is finite-dimensional).
This vector space carries a (possibly projective) representation of the mapping class group of $\Sigma$.
The assignment is local in the sense that gluing corresponds
at least
in easy cases
to a summation over all possible labels.
We remind the reader of the  definition of a modular functor in Section~\ref{secmf}.

Once one has finite-dimensional mapping class group representations, one may ask 
about the faithfulness of such mapping class group representations.
The question whether mapping class groups have faithful finite-dimensional representations is known as the \emph{linearity problem for mapping class groups}~\cite{birman}.
We address in this article an important sub-aspect of the faithfulness question: We 
focus on the action of Dehn twists on the spaces of conformal blocks of modular functors.
Dehn twists are the most basic type of mapping classes; they have  infinite order.
Understanding their action on spaces of conformal blocks can already be quite subtle.
The most classical construction for modular functors uses as an input datum a \emph{modular fusion category}, a linear abelian monoidal category with a non-degenerate braiding (only finite direct sums of the unit trivially double braid with all objects) and a ribbon structure, subject to the assumption of finiteness (the category is linearly equivalent to finite-dimensional modules over some algebra) and, most importantly, semisimplicity. In that situation, one does not only obtain a modular functor, but a 
 three-dimensional topological field theory via the Reshetikhin-Turaev construction~\cite{rt1,rt2}.
For this construction, the Dehn twists act always by automorphisms of finite order by Vafa's Theorem, a statement going back to~\cite{vafa,andersenmoore,baki}, see \cite[Theorem~5.1]{etingoftwist} for a refinement of the original statement.
This might seem a little disappointing at first because it clearly prevents faithfulness. Nonetheless, the 
mapping class group representations obtained from this classical construction, sometimes referred to as \emph{quantum representations of mapping class groups}, still turn
 out to be extremely interesting: With appropriately chosen input datum, these representations are \emph{asymptotically faithful} by results of Andersen~\cite{andersenfaithful} and Freedman-Walker-Wang~\cite{fww}
which means, roughly, that one can exhibit a certain infinite family of modular functors that, when considered
together, detect all mapping classes modulo center.

If one drops from the definition of a modular fusion category the extremely strong and restrictive condition of semisimplicity, 
one obtains the notion of a \emph{modular category}. With a modular category as input datum, one can still build a modular functor by work of Lyubashenko~\cite{lyubacmp,lyu,lyulex}.
This construction relies crucially on understanding a certain coend Hopf algebra inside the modular category. The intricate algebra behind this makes concrete computations for this modular functor challenging. 
One of the key features of this construction is that, without the semisimplicity, the restriction on the order of the Dehn twists in the representations disappears, and in fact it has been shown that in certain examples certain Dehn twists can act by 
automorphisms of infinite order~\cite{bcgp}. Even more recently, it was shown in \cite[Section~5.2]{gai2}
that for the modular functor built from the small $\operatorname{sl}_2$ quantum group all Dehn twists of closed surfaces 
act by automorphisms of infinite order.

In this short article, our goal is to calculate the action of Dehn twists and in particular their order in the representation in the greatest generality possible. In particular, we want to arrive  at a clear distinction  between statements that hold generally 
 and those statements that are tied to specific examples.

One of the key motivations for our study of Dehn twists is that the 
order of the Dehn twist in the representations coming from modular functors
could be one of the most
 important properties when it comes to deciding on whether 
or not these representations are faithful.
In other words, Dehn twists acting by finite order 
automorphisms could be the main obstruction to faithfulness.
Evidence for this idea that is shared by many experts includes the following:
By a result of Schauenburg-Ng~\cite{schauenburgng}
for the modular functor of a modular fusion category
the representation of the mapping class group 
$\catf{SL}(2,\mathbb{Z})$ of the torus has as its kernel the congruence 
subgroup whose level is the order of any Dehn twist under the representation (which is always finite here because of semisimplicity).
This  applies to the semisimple case and genus one, but it does seem
 to indicate that the finite order of the Dehn twist in the representation is the main obstruction to faithfulness.
At least in certain non-semisimple examples in which the Dehn twist acts by an infinite order element,
the $\catf{SL}(2,\mathbb{Z})$-representation becomes indeed faithful modulo center~\cite[Theorem~1.4]{bcgp}.
Another very interesting and recent result (in the semisimple case) connecting directly the order of Dehn twists and the kernel of (examples of)
quantum representations of mapping class group is given by Detcherry-Santharoubane~\cite{renaudramanujan}.

The Dehn twist action on the spaces of conformal blocks comes by definition
 from the ribbon structure of the modular category, a natural automorphism of the identity. As such, the ribbon structure has an order. The first idea is that the order of the ribbon structure translates `loss-free' to the order of the Dehn twists in the representation, so that the entire construction would have some intrinsic faithfulness, at least on the Dehn twist level. 
 We prove that for Dehn twists about \emph{non-separating} essential simple closed curves this is indeed the case:
 
 \begin{reptheorem}{thmmoddehn}
 	Let $\cat{A}$ be a modular category 
 	and denote by
 	$\mathfrak{F}_{\! \cat{A}}$ its modular functor.
 	Let $d$ be
 	any Dehn twist 
 	about a non-separating essential simple closed curve
 	on a closed surface $\Sigma$
 	with genus $g\ge 1$.
 	Then $d$ acts on the space of conformal blocks $\mathfrak{F}_{\! \cat{A}}(\Sigma)$ by a linear automorphism $\mathfrak{F}_{\! \cat{A}}(d)$ whose order
 	in $\PGL(\mathfrak{F}_{\! \cat{A}}(\Sigma))$
 	is equal to the order $|\theta|$ of the ribbon structure of $\cat{A}$.
 	In particular, $\mathfrak{F}_{\! \cat{A}}(d)$ has infinite order if and only if the ribbon structure has infinite order.
 	\end{reptheorem}

For Dehn twists about \emph{separating} simple closed curves, the ribbon structure generally
cannot translate into automorphisms of the same order. This already follows from results of Fjelstad and Fuchs~\cite{ff}:
The Dijkgraaf-Witten type modular functor
for a finite abelian group $G$ assigns to a surface $\Sigma$
the free vector space generated by $H^1(\Sigma;G)$, equipped with the geometric action of the mapping class group, so that the Torelli subgroup  of the mapping class group (in which the Dehn twists about separating simple closed curves are contained) acts trivially. This happens even though the modular category behind the Dijkgraaf-Witten modular functor (the representation category of the Drinfeld double of $G$) has a non-trivial ribbon structure. As a result, Theorem~\ref{thmmoddehn} cannot hold in the separating case. Instead, we find:

\begin{reptheorem}{thmmoddehn2}
	Let $\cat{A}$ be a modular category 
	and denote by
	$\mathfrak{F}_{\! \cat{A}}$ its modular functor.
	Let $d$ be
	any Dehn twist 
	about a separating essential simple closed curve
	on a closed surface $\Sigma$
	with genus $g\ge 2$ that separates the surface into pieces of genus $g'$ and $g''$.
	Then $d$ acts on the space of conformal blocks $\mathfrak{F}_{\! \cat{A}}(\Sigma)$ by a linear automorphism $\mathfrak{F}_{\! \cat{A}}(d)$ whose order
	in $\PGL(\mathfrak{F}_{\! \cat{A}}(\Sigma))$
	is equal to $\min \{  | \theta_{\mathbb{A}^{\otimes g'}}| ,  | \theta_{\mathbb{A}^{\otimes g''}}|   \}$,
	where $\mathbb{A}=\int_{X \in \cat{A}} X \otimes X^\vee\in\cat{A}$ is the canonical end of $\cat{A}$.
\end{reptheorem}

In the last section, we present some applications, special cases and examples of our two main results: In 
Corollary~\ref{corlie}, we give some implications of 
Theorem~\ref{thmmoddehn} to the existence of unitary structures
on spaces of conformal blocks.
In Proposition~\ref{propzg}, we generalize Theorem~\ref{thmmoddehn} to a faithfulness result for the representations of  free abelian subgroups of the mapping class group generated by certain families of commuting Dehn twists.
A criterion for when exactly the representations annihilate the Johnson kernels is given in Proposition~\ref{propjohnson}. A similar result for the Torelli groups is stated in Proposition~\ref{proptorelli}. The latter can be seen as a generalization of a result by Fjelstad-Fuchs~\cite{ff} from Drinfeld doubles to arbitrary modular categories. 
	
	\vspace*{0.2cm}\textsc{Acknowledgments.} We thank
Adrien Brochier, 
Jürgen Fuchs,	
Simon Lentner,
Peter Schauenburg,	
Christoph Schweigert and
Alexis Virelizier for helpful discussions related to this project and 
Ehud Meir for sketching an idea on how to prove Lemma~\ref{lemmapropmor}.
LM gratefully acknowledges support of the Simons Collaboration on Global Categorical Symmetries. Research at Perimeter Institute is supported in part by the Government of Canada through the Department of Innovation, Science and Economic Development and by the Province of Ontario through the Ministry of Colleges and Universities. The Perimeter Institute is in the Haldimand Tract, land promised to the Six Nations.
LW  gratefully acknowledges support
by the ANR project CPJ n°ANR-22-CPJ1-0001-01 at the Institut de Mathématiques de Bourgogne (IMB).
The IMB receives support from the EIPHI Graduate School (ANR-17-EURE-0002).

\section{Building modular and ansular  functors\label{secmf}}
This section contains a very condensed reminder on how modular functors can be defined
and constructed, and how a lot of the information
can be read off from the underlying \emph{ansular functor}, i.e.\ the restriction of the modular functor to handlebodies.

Defining modular functors is a non-trivial task. 
The definitions in \cite{Segal,ms89,turaev,tillmann,baki} 
pertain mostly to the semisimple case and therefore do not really cover 
the class of examples provided in~\cite{lyubacmp,lyu,lyulex}. 
	A careful setup of all the relevant notions adapted to this class of examples is given in~\cite{jfcs}. 
	A general definition of the notion of a modular functor using the language of modular operads in the sense of Getzler-Kapranov~\cite{gkmod} using the graph description of Costello~\cite{costello} is given in~\cite{cyclic,brochierwoike}: A \emph{modular functor with values in a symmetric monoidal bicategory $\cat{S}$} is a modular algebra over (a certain extension of) the modular operad $\Surf$ of surfaces. 
	The extension will account for the fact that the mapping class group representations are generally only projective. This is due to the framing anomaly, see \cite{atiyahframing} for the geometric background and \cite[Section~3]{gilmermasbaum} for a discussion in the context of mapping class group representations. 
	A general treatment in the language of operads is given in \cite[Section~3.1]{brochierwoike}.
	For this article, the framing anomaly will be hardly relevant. Therefore, we will omit it from this short reminder section.

	The full definition of a modular functor is long and tedious, simply because 
	the definition of a modular algebra over a category-valued operad with values in a symmetric monoidal bicategory $\cat{S}$
	involves a lot of coherence isomorphisms. 
	For the purpose of this article, it will suffice to partially spell out the definition 
	in the relevant case where $\cat{S}$ is the symmetric monoidal bicategory $\Rexf$ of
	\begin{itemize}
		\item finite $k$-linear categories in the sense of \cite{etingofostrik}\label{pagefinitecategories}
		(linear abelian categories linearly equivalent to finite-dimensional modules over some finite-dimensional $k$-algebra over our fixed algebraically closed field $k$),
	\item right exact functors (functors preserving finite colimits),
	\item linear natural transformations,
	\end{itemize}
	with the Deligne product as monoidal product whose monoidal unit is the category $\vect$ of finite-dimensional
	vector spaces over $k$.
	A reference discussing $\Rexf$ and its dual variant $\Lexf$ (with left exact functors replacing the right exact ones) is \cite{fss}. This article also discusses (co)ends in finite linear categories a basic knowledge of which will be assumed in the sequel.

	A $\Rexf$-valued modular functor $\mathfrak{F}$
	has an underlying category $\cat{A}\in\Rexf$ (one can think of this as the value of the modular functor on the circle)
	and, for each surface $\Sigma$
	(for us, this will mean throughout a compact oriented two-dimensional manifold with parametrized, possibly empty boundary) with $n \ge 0$ boundary components, a right exact functor $\mathfrak{F}(\Sigma):\cat{A}^{\boxtimes n}\to \vect$
	with an action of the mapping class group $\Map(\Sigma):=\pi_0( 
	\Diff(\Sigma))$ of $\Sigma$. 
	Here $\Diff(\Sigma)$ is the topological group of diffeomorphisms $\Sigma \to \Sigma$ preserving the orientation and the boundary parametrization; we refer to \cite{farbmargalit} for a textbook reference on mapping class groups.
	This means that a mapping class $f \in \Map(\Sigma)$ acts by an automorphism of $\mathfrak{F}(\Sigma)$.
	In particular, for $X_1,\dots,X_n \in \cat{A}$ (to be understood as boundary labels), the vector space $\mathfrak{F}(\Sigma;X_1,\dots,X_n)$ defined as the image of $X_1\boxtimes\dots\boxtimes X_n\in\cat{A}^{\boxtimes n}$ under $\mathfrak{F}(\Sigma):\cat{A}^{\boxtimes n}\to\vect$ carries an action of $\Map(\Sigma)$.  This vector space is the space of conformal blocks for the surface with these boundary labels, thereby connecting to the intuitive description of modular functors outlined in the introduction.
	The compatibility with gluing is, in this generality, a little tedious to formulate; we refer to \cite[Section~2]{cyclic} for details: It uses a symmetric non-degenerate pairing $\kappa : \cat{A}\boxtimes \cat{A}\to \vect$. 
	Non-degeneracy means that there is a copairing or coevaluation $\Delta : \vect \to \cat{A}\boxtimes \cat{A}$, i.e.\ an object in $\cat{A}\boxtimes\cat{A}$, that together with $\kappa$ satisfies the zigzag identities up to isomorphism. The symmetry of the pairing means that coherent isomorphisms $\kappa(X,Y)\cong\kappa(Y,X)$ are part of the data.
	If we now take a surface $\Sigma$ and glue a pair of boundary components together and call the result $\Sigma'$, we have an isomorphism
	$\mathfrak{F}(\Sigma';-)\cong\mathfrak{F}(\Sigma; \dots , \Delta , \dots)$, where $\Delta$ is inserted in exactly those two slots that are sewn together.
	Again, let us emphasize that we are suppressing here
	 a lot of coherence data and compatibilities. 
	 We will discuss the part of the gluing behavior relevant to us below in Lemma~\ref{lemmablockforH}.

A modular functor can be restricted to genus zero. We then obtain an algebra over the framed $E_2$-operad because the framed $E_2$-operad is equivalent to the operad of oriented genus zero surfaces. By \cite{WahlThesis,salvatorewahl} this endows $\cat{A}$ with \begin{itemize}
	\item a monoidal product $\otimes:\cat{A}\boxtimes\cat{A}\to\cat{A}$ with unit $I$, \item a braiding $c$, i.e.\ natural isomorphisms $c_{X,Y}:X\otimes Y \ra{\cong} Y \otimes X$ for all $X,Y\in\cat{A}$ subject to the hexagon axioms,
	\item and a balancing $\theta$, i.e.\ a natural automorphism $\theta_X : X\ra{\cong} X$ with $\theta_I=\id_I$ and \begin{align}\label{eqnbalancing} \theta_{X\otimes Y}=c_{Y,X}c_{X,Y}(\theta_X\otimes\theta_Y) \end{align} for all $X,Y\in\cat{A}$.
	 \end{itemize}
Moreover, the framed $E_2$-operad is \emph{cyclic} in the sense of Getzler-Kapranov~\cite{gk} and $\cat{A}$ becomes a cyclic algebra over this cyclic operad. By the main result of \cite{cyclic} this amounts to a \emph{ribbon Grothendieck-Verdier structure} on $\cat{A}$ in the sense of Boyarchenko-Drinfeld~\cite{bd}. (We should warn the reader that in \cite{cyclic,mwansular} the target category $\Lexf$ is used instead. As a consequence, some conventions will be dual. This is however not a significant issue.)
	This means that $\cat{A}$ comes additionally with an equivalence $D:\cat{A}\to\cat{A}^\opp$
	 that makes the functors $\cat{A}(X\otimes -,K)$ with $K=DI$ representable via $\cat{A}(X\otimes -,K)\cong \cat{A}(-,DX)$, where $\cat{A}(-,-)$ denotes the morphism spaces of $\cat{A}$. (What we call $D$ in the sequel would be $D^{-1}$ in the conventions of~\cite{bd}.)
	 Moreover, the relation $\theta_{DX}=D\theta_X$ holds for all $X\in\cat{A}$. The duality functor $D$ is related to 
	  the pairing via $\kappa(X,Y)\cong \cat{A}(Y,DX)^*$.
	 A rich source of ribbon Grothendieck-Verdier categories are vertex operator algebras~\cite{alsw}.

	 A special case of great relevance to this paper is the situation in which $D$ is a \emph{rigid} duality.\label{refrigidduality} This means that $D$ sends any object $X$ to its rigid dual $X^\vee$ that comes with an evaluation $X^\vee \otimes X \to I$ and coevaluation $I\to X \otimes X^\vee$ satisfying the usual zigzag identities. A ribbon Grothendieck-Verdier category with rigid duality and simple unit is called a \emph{finite ribbon category} in the sense of \cite{etingofostrik,egno}; without the braiding and the balancing, one would just call it a \emph{finite tensor category}.
	 The balancing of a finite ribbon category is also called 
	 \emph{ribbon twist} or \emph{ribbon structure}.
	 Thanks to rigidity, the monoidal product of a finite tensor category is always exact.
	 Note also that, thanks to the unit being simple (meaning that it has two non-isomorphic quotients), a finite tensor category cannot be the zero linear category.\label{pageftcnonzero}

	 A \emph{modular category} is a finite ribbon category whose braiding is \emph{non-degenerate} in the sense that $c_{Y,X}c_{X,Y}=\id_{X\otimes Y}$ for all $Y\in\cat{A}$ implies that $X$ is a finite direct sum $I\oplus \dots \oplus I$ of copies of 	the unit.
	 A source of finite ribbon categories are ribbon Hopf algebras, see \cite[XIV.6]{kassel} for a textbook reference.
	 Such a category will even be modular if the ribbon Hopf algebra is factorizable, see e.g.~\cite[Section 2.2]{svea2}.

	 Another condition that we will encounter below is \emph{unimodularity}:
	 A finite tensor category $\cat{A}$ is called \emph{unimodular}
	 if 
	  the distinguished invertible object $\alpha \in \cat{A}$ from~\cite{eno-d} controlling the quadruple dual via $-^{\vee\vee\vee\vee}\cong \alpha \otimes-\otimes\alpha^{-1}$ is isomorphic to the monoidal unit. By \cite[Proposition~4.5]{eno-d} this condition is fulfilled automatically for a modular category.
	
	Since the genus zero part of a modular functor amounts by the results mentioned above exactly to a ribbon Grothendieck-Verdier category, it is natural to ask whether any ribbon Grothendieck-Verdier category gives us a uniquely determined modular functor.
	With regards to that question, the following can be said:
	
	\begin{itemize}
		\item By \cite{cyclic,mwansular} a ribbon Grothendieck-Verdier category extends uniquely to a modular algebra over the modular handlebody operad, a so-called \emph{ansular functor}.
		For the definition of an ansular functor, one replaces in the definition of a modular functor sketched above all surfaces with three-dimensional handlebodies (and therefore mapping class groups with handlebody groups; see \cite{henselprimer} for an introduction to handlebody groups). The proof uses the modular envelope of cyclic operads defined by Costello~\cite{costello} and results by Giansiracusa~\cite{giansiracusa}, in combination with \cite{mwdiff}. Every modular functor restricts to an ansular functor along the map which sends a handlebody to its boundary.   
		
		\item In order to have an extension to higher genus, an additional condition formulated in terms of factorization homology, a generalization of skein theory, needs to be satisfied~\cite{brochierwoike}. Then the extension is unique up to a contractible choice.
		For this article, the only relevant information is that a modular category in fact satisfies this condition; it gives us a modular functor, namely the one built by Lyubashenko~\cite{lyubacmp,lyu,lyulex}.
		This justifies it to speak of `the modular functor of a modular category $\cat{A}$'; we denote it by $\mathfrak{F}_{\! \cat{A}}$.
		\end{itemize}
	Given that we will mostly consider the modular functor of a modular category, one could ask why
	we recall in this section ribbon Grothendieck-Verdier categories and ansular functors.
	The reason for this is that the action of Dehn twists, which is the main concern of this article,
	can be conveniently described in terms of the ansular functors. 
	This turns out to be much simpler than using directly Lyubashenko's construction.

\section{The generalized ribbon element}

Let $\cat{A}$ be a ribbon Grothendieck-Verdier category in $\Rexf$. 
It comes in particular with a balancing $\theta$ whose order, as a natural automorphism of the identity of $\cat{A}$,
 we denote by $|\theta|\in \mathbb{N} \cup \{\infty\}$. 
We denote the associated ansular functor by $\widehat{\cat{A}}$.

For many of the subsequent computations, we will fix a projective generator $G$ for $\cat{A}$. In other words, $\cat{A}$ is linearly equivalent to the category of finite-dimensional modules over the finite-dimensional $k$-algebra $B:= \End_\cat{A}(G)$. Obtaining such a projective generator is possible because the underlying category of $\cat{A}$ is supposed to be finite in the sense explained on page~\pageref{pagefinitecategories}. We should however warn the reader that generally the monoidal structure does not have to come from a bialgebra structure on $B$.

The ansular functor associated to $\cat{A}$ sends the cylinder labeled with $G$ (once as incoming, once as outgoing label) to the vector space $B^*=\End_\cat{A}(G)^*$.
 The Dehn twist of the cylinder acts on $B^*$ through the linear dual of postcomposition or precomposition with $\theta_G:G\to G$.
 (All of this is tautologically true by the construction of $\widehat{\cat{A}}$ from $\cat{A}$; see also \cite[Remark~5.10]{mwansular}.)
 
 We denote by $\nu := \theta_G \in B=\End_\cat{A}(G)$ the component of the balancing at the projective generator $G$; we call $\nu$ the \emph{generalized ribbon element of $\cat{A}$ (and this choice of generator)}.
We continue with a basic observation:

\begin{proposition}\label{proporder}
	For a ribbon Grothendieck-Verdier category $\cat{A}$ in $\Rexf$, we have $\nu \in Z(B)$ and
	\begin{align}|\theta| = |\nu| \ . 
		\end{align} 
	This number is a uniform bound for the images of all Dehn twists
	under the handlebody group representations provided by the ansular functor $\widehat{\cat{A}}$.
	\end{proposition}

\begin{proof}
	The fact that $\nu$ is central follows directly from naturality. Moreover, it is clear that $|\theta|$ is a
	uniform bound for the images of all Dehn twists
	under the handlebody group representations provided by the ansular functor $\widehat{\cat{A}}$, simply because the Dehn twists locally always act by the balancing.
	
Clearly, $ |\nu| \le |\theta|$ by definition. 
To prove the non-trivial inequality $|\theta|\le |\nu|$, observe that, since $\cat{A}$ is in particular 
a linear abelian category with enough projective objects, there is for each $X\in\cat{A}$ a short exact sequence $P\ra{p} X\to 0$ with a projective object $P\in\cat{A}$. Since $P$ is projective and $G$ a projective generator, $P\oplus Q \cong G^{\oplus n}$,
where $G^{\oplus n}$ is the $n$-fold direct sum of $G$ for some $n\ge 1$ and $Q$ is some object in $\cat{A}$.
Consider now the epimorphism $\pi : G^{\oplus n} \cong P\oplus Q \ra{\text{proj}} P\ra{p} X$. 
	With the naturality of $\theta$, we find
	$\theta_X \pi = \pi  \theta_{G^{\oplus n}}$
	and $\theta_{G^{\oplus n}} = (\theta_G)^{\oplus n}$, which gives us $	\theta_X \pi = \pi (\theta_G)^{\oplus n}=\pi \nu^{\oplus n}$ and therefore also
	\begin{align}
		\theta_X^\ell \pi = \pi (\theta_G^\ell)^{\oplus n}=\pi (\nu^\ell)^{\oplus n} \quad \text{for all} \quad \ell \ge 0 \ . \label{eqnnattheta}
		\end{align}
This implies $|\theta|\le |\nu|$, since $\pi$ is an epimorphism,
 and finishes the proof.
	\end{proof}

\section{The action of a Dehn twist about a non-separating simple closed curve}
In this section, we calculate the action of Dehn twists about separating simple closed curves on spaces of conformal blocks of a modular category.
The statement will actually follow from a more technical statement about the action of Dehn twists about meridians of handlebodies
on the spaces of conformal blocks for ansular functors.
Meridians are essential simple closed curves in the boundary of the handlebody that bound a disk in the handlebody, see~\cite[Section~1.1]{henselprimer} for the terminology.

\begin{theorem}\label{thmmain}
	Let $\cat{A}$ be a unimodular
	finite ribbon category in $\Rexf$ and $H_g$ the handlebody of genus $g\ge 1$ with no embedded disks in its boundary. 
	The Dehn twist $d$ about any non-separating meridian $\gamma$ acts on the vector space $\widehat{\cat{A}}(H_g)$ 
	associated to $H_g$ by the ansular functor $\widehat{\cat{A}}$
	for $\cat{A}$ by an automorphism whose order is $|\theta|$
	in $\PGL(\widehat{\cat{A}}(H_g))$. 
\end{theorem}

\begin{remark}
Dehn twists about meridians generate the twist subgroup $\catf{Tw}(H_g)\subset \Map(H_g)$ that fits into the short exact sequence
\begin{align}
	1 \to \catf{Tw}(H_g)\to\Map(H_g)\to\catf{Out}(F_g)\to 1 \ , 
	\end{align}
where $\catf{Out}(F_g)$ is group of outer automorphisms of the free group on $g$ generators, see~\cite{luft}. 
Theorem~\ref{thmmain} tells us therefore that all the representations $\widehat{\cat{A}}(H_g)$
of $\Map(H_g)$ descend to a representation of $\catf{Out}(F_g)$ if and only if $\theta$ is 
 the identity.
Note that the latter would always imply that the braiding of $\cat{A}$ is symmetric.
\end{remark}

We now prepare the proof of Theorem~\ref{thmmain}:
Let $\mathbb{A}=\int_{X \in \cat{A}} X \otimes X^\vee\in \cat{A}$ be the canonical end of $\cat{A}$
(the Grothendieck-Verdier duality $D$ of $\cat{A}$ is by definition the rigid duality $-^\vee$, as explained on page~\pageref{refrigidduality}). 
We refer to \cite[Section~2.2]{fss}
for an introduction to (co)ends in finite linear categories. We will use the convention 
 $X^{\otimes 0}=I$ for every object $X$.

For a $k$-algebra $T$ and a $T$-bimodule $M$, we denote by
\begin{align}
Z(T;M):= \{m\in M \, | \, t.m=m.t \ \text{for all}\ t\in T\} \subset M
\end{align}
the center of $M$ relative $T$ (also known as the zeroth Hochschild cohomology $HH^0(T;M)$
 of $T$ with coefficients in $M$).
If $M$ is $T$ with left and right regular $T$-action, then $Z(T;T)=Z(T)$, the center of the algebra $T$.
Note that a $T$-bimodule $M$ can be seen as a functor $M: (\star \DS T)^\opp \otimes \star \DS T\to\Vect_k$, where $\star \DS T$ is the $k$-linear category with one object $\star$ whose endomorphism algebra is $T$.
Then $Z(T;M)$ can be expressed as the end
\begin{align}
 Z(T;M)\cong
 	\int_{\star \DS T}M \ .          \label{eqnexpend}
	\end{align}

\begin{lemma}\label{lemmablockforH}
	Let $\cat{A}$ be a unimodular finite ribbon category. Then
	\begin{align}
		\widehat{\cat{A}}(H_{g}) \cong Z\left(B; \cat{A}\left(G,G\otimes \mathbb{A}^{\otimes (g-1)}\right)\right) ^*\end{align}
		 for $g\ge 1$, and the Dehn twist about the
		 non-separating meridian $\gamma$ from
		  Theorem~\ref{thmmain}
		  acts by $\nu=\theta_G$ on the $G$ in the left argument of the hom.
	\end{lemma}

Lemma~\ref{lemmablockforH}, that we will prove using the excision property of spaces of conformal blocks, should be understood in the way that, after cutting $H_g$ at $\gamma$, we obtain  a handlebody of genus $g-1$ with two embedded disks in its boundary, 
such that, after labeling both these disks by $G$, we obtain
 a $B$-bimodule $\cat{A}\left(G,G\otimes \mathbb{A}^{\otimes (g-1)}\right)$
whose relative center (zeroth Hochschild cohomology) is $\widehat{\cat{A}}(H_{g})^*$. This is essentially the perspective on excision from~\cite{dmf}, but formulated here in a different context and, most importantly, equivariantly with respect to a Dehn twist action.

\begin{proof}[\slshape Proof of Lemma~\ref{lemmablockforH}]
	 We can obtain the value of the ansular functor $\widehat{\cat{A}}$ on $H_g$  by the  excision result for the spaces of conformal blocks of ansular functors; we refer to~\cite[Theorem 6.4]{cyclic} (there the theory is set up using left exact functor instead of right exact ones as in this paper; see \cite[Section~8.1]{brochierwoike} for the version using right exact functors). We then find
	\begin{align}
	\widehat{\cat{A}}(H_g)&\cong \cat{A}(\mathbb{A}^{\otimes g},I)^*\cong \cat{A} \left(   \int_{X\in\cat{A}} X \otimes X^\vee \otimes \mathbb{A}^{\otimes (g-1)},I      \right)^* \cong
	\cat{A} \left(   \int_{X\in\cat{A}} X^\vee \otimes X \otimes \mathbb{A}^{\otimes (g-1)},I      \right)^*
	\end{align}
	In the last step, we move the $\vee$ from the right to the left dummy variable using the pivotal structure.
	From the construction of ansular functors, it follows also that, under these isomorphisms, the Dehn twist about $\gamma$ acts on $\cat{A} \left(   \int_{X\in\cat{A}} X^\vee \otimes X \otimes \mathbb{A}^{\otimes (g-1)},I      \right)^*$
	by $\theta_{X^\vee}=\theta_X^\vee$ on the $X^\vee$ under the end. Thanks to the universal property of the end, we can say equivalently that the Dehn twist about $\gamma$ acts by $\theta_{X}$ on the $X$.
	
Let us now simplify this expression:	We obtain $\int_{X\in\cat{A}} X^\vee \otimes X$ by applying $\otimes ((-)^\vee \boxtimes \id_{\cat{A}})$ to $\int_{X\in\cat{A}} X \boxtimes X\in \cat{A}^\opp \boxtimes \cat{A}$.
	Now we use the isomorphism $\int_{X\in\cat{A}} X \boxtimes X\cong \int^{X \in \cat{A}} X \boxtimes \nakal X$ in $\cat{A}^\opp \boxtimes \cat{A}$
	 with the left Nakayama functor $\nakal$, which
	  here is given by $\nakal = \int_{Y\in\cat{A}} \cat{A}(Y,-)\otimes Y \cong - \otimes \alpha$ with the distinguished invertible object $\alpha$ of $\cat{A}$~\cite[eq.~(3.52) \& Lemma~4.10]{fss}; note that this makes use of the pivotal structure. 
	 The isomorphism
$	\int^{X \in \cat{A}} X \boxtimes \nakal X\cong  \int_{X\in\cat{A}} X \boxtimes X$
is induced
by the maps \begin{align} X \boxtimes \nakal X \ra{\substack{\text{structure map of} \\ \text{the end defining $\nakal$}}}
X \boxtimes \cat{A}^\opp (X,Y)\otimes Y \ra{\text{evaluation}} Y \boxtimes Y \ . 
\end{align}
Therefore, the automorphism of $	\int^{X \in \cat{A}} X \boxtimes \nakal X$
induced by $\theta_X \boxtimes \id_{\nakal X}$ corresponds to the automorphism of $ \int_{X\in\cat{A}} X \boxtimes X$ induced by $\theta_X \boxtimes \id_X$. 
This implies
	$\int_{X\in\cat{A}} X^\vee \otimes X\cong \int^{X \in \cat{	A}} X^\vee \otimes \nakal X$, and the automorphisms induced by $\theta_{X^\vee}$ on the 
	 $X^\vee$ correspond to each other under this isomorphism.
	The finite ribbon category $\cat{A}$ being unimodular means precisely $\alpha \cong I$.
	Using the duality once more, we therefore arrive at
	\begin{align}
	\widehat{\cat{A}}(H_g)\cong \cat{A} \left(   \mathbb{A}^{\otimes (g-1)},\int_{X \in \cat{A}} X^\vee \otimes X      \right)^*
	\end{align}
	We denote by  $G // B$  the full subcategory of $\cat{A}$ whose single object is our projective generator $G$.
	In other words, $G//B$ has one object with endomorphism algebra $B$. 
	Then by \cite[Proposition~5.1.7]{kl} $\int_{X \in \cat{A}} X^\vee \otimes X\cong \int_{X \in G//B} X ^\vee \otimes X$
	and, since $\cat{A}(\mathbb{A}^{\otimes (g-1)},-)$ preserves ends,
	\begin{align}
	\widehat{\cat{A}}(H_g)^* \cong \int_{X \in G//B} \cat{A} \left(  X \otimes  \mathbb{A}^{\otimes (g-1)}, X      \right)
	\cong \int_{X \in G//B} \cat{A} \left(  X  , X   \otimes \mathbb{A}^{\otimes (g-1)}   \right) \ ,  \end{align}
	where we have also used that unimodularity  implies the self-duality of $\mathbb{A}$~\cite[Theorem~4.10]{shimizuunimodular}. 
	Finally, by~\eqref{eqnexpend}
	\begin{align}
	\int_{X \in G//B} \cat{A} \left(  X  , X   \otimes \mathbb{A}^{\otimes (g-1)}   \right) =Z\left(B; \cat{A}\left(G,G\otimes \mathbb{A}^{\otimes (g-1)}\right)\right) \ . 
	\end{align}
	From the description of the action of the Dehn twist about $\gamma$ on $\cat{A} \left(   \int_{X\in\cat{A}} X \otimes X^\vee \otimes \mathbb{A}^{\otimes (g-1)},I      \right)^*$ given above, it follows that the action of the same Dehn twist on $Z\left(B; \cat{A}\left(G,G\otimes \mathbb{A}^{\otimes (g-1)}\right)\right)$ is through the action of $\nu=\theta_G$ on the $G$ in the left slot of the hom in $\cat{A}\left(G,G\otimes \mathbb{A}^{\otimes (g-1)}\right)$.
	\end{proof}

\begin{corollary}\label{corogamma}
	The order of the action of the Dehn twist about $\gamma$ on $	\widehat{\cat{A}}(H_{g}) $ is the order $o_\gamma$ 
	of the action of $\nu$ on $Z\left(B; \cat{A}\left(G,G\otimes \mathbb{A}^{\otimes (g-1)}\right)\right)$. 
	\end{corollary}

	One of the algebraic ingredients for the proof of Theorem~\ref{thmmain} is the following Lemma that might be well-known, but whose proof we spell out in lack of a reference. We thank Ehud Meir for providing helpful comments on this statement as well as ideas for its proof. 
	
	\begin{lemma}\label{lemmapropmor}
		Let $\cat{A}$ be a finite category and $X,Y \in \cat{A}$.
		If $X$ is simple, then the evaluation map
		\begin{align}\varepsilon_{X,Y}: \cat{A}(X,Y)\otimes X \to Y
		\end{align}
		is a monomorphism.
	\end{lemma}
	
	If $X$ is not simple, the statement is wrong: Then the map $\cat{A}(X,X)\otimes X \to X$ can generally not be a monomorphism for dimension reasons.

	\begin{proof}
	For the proof, we identify $\cat{A}$ with finite-dimensional modules over $B=\End_\cat{A}(G)$, where $G$ is a projective generator.
		Let $S$ denote the image of ${\varepsilon}_{X,Y}$ and $\widetilde{\varepsilon}_{X,Y} \colon \cat{A}(X,Y)\otimes X \to S$ the epimorphism obtained by restriction in range; put differently,  $S = \sum_{f \in \cat{A}(X,Y)} f(X)\subset Y$. With a basis $b_1, \dots , b_n $ of $ \cat{A}(X,Y)$, we have $S=\sum_{i=1}^n b_i(X)$, thereby proving that there is an epimorphism $X^{\oplus n} \to S$. As a quotient of the semisimple object $X^{\oplus n}$, the object $S$ is also semisimple, i.e.\ $S=\bigoplus_{j=1}^m S_j$ with simple objects $S_1 , \dots, S_m$. We have $\cat{A}(X,S_j) \neq 0$ for all $j=1,\dots,m$ because otherwise $X^{\oplus n} \to S$ could not be an epimorphism. By Schur's Lemma this implies $X\cong S_j$ for $1\le j\le m$
		 and hence
		$S\cong X^{\oplus m}$. 
		Thanks to $\cat{A}(X,S) \cong k^m$,
		the map $\varepsilon_{X,S} : \cat{A}(X,S) \otimes X \to S$ can now be identified with the 
		isomorphism $k^m \otimes X \cong X^{\oplus m}\cong S$. 
		
		Consider now the commutative diagram
		\begin{equation}
			\begin{tikzcd}
				\cat{A}(X,Y)\otimes X	 \ar[rrrdd,"\widetilde{\varepsilon}_{X,Y}"]      \\  \\
				\cat{A}(X,S) \otimes X  \ar[rrr,swap,"\varepsilon_{X,S}"] \ar[uu,"       S \subset Y    "]   &&&      S \ . 
			\end{tikzcd}
		\end{equation}
		The vertical map induced by $S\subset Y$ is a monomorphism because $\cat{A}(X,-)$ is left exact,
		and the tensoring of objects in $\cat{A}$ with vector spaces is exact (therefore both functors
		preserve monomorphisms).  Since all maps $X \to Y$ factor by definition
		 through $S$, the map $\cat{A}(X,S) \otimes X \to \cat{A}(X,Y)\otimes X$ is also an epimorphism and hence an isomorphism. This implies that $\widetilde{\varepsilon}_{X,Y}$ is an isomorphism and concludes the proof.
	\end{proof}

	\begin{lemma}\label{lemmahommap}
		Let $\cat{A}$ be a monoidal category in $\Rexf$ with exact monoidal product and simple unit.
		Then the map
			\begin{align} \cat{A}(I,X) \otimes \cat{A}(I,Y) \to \cat{A}(I,X\otimes Y) \quad \text{for}\quad X,Y \in \cat{A}\label{eqnimagemap}
			\end{align} induced by the monoidal product $\otimes$
			is a monomorphism.
	\end{lemma}

	\begin{proof}
		The map $\cat{A}(I,X) \otimes I \to X$ is a monomorphism by Lemma~\ref{lemmapropmor}.
		Since $-\otimes Y$ is exact, $\cat{A}(I,X) \otimes Y \to X\otimes Y$ is a monomorphism.
		Its image~\eqref{eqnimagemap} under the left exact functor $\cat{A}(I,-)$ is a monomorphism as well.
		\end{proof}
	
	\spaceplease	
	\begin{proposition}\label{prophomnonzero}
		Let $\cat{A}$ be a finite tensor category.
		\begin{pnum}
			\item The vector space $\cat{A}(I,\mathbb{A}^{\otimes g})$ is non-zero for all $g \ge 0$. \label{prophomnonzeroi}
			\item 
			The morphism \label{prophomnonzeroii}
			\begin{align}
				\cat{A}(I,\mathbb{A}^{\otimes (g-1)}) \otimes G \to \mathbb{A}^{\otimes (g-1)} \otimes G
			\end{align}
			  is a non-zero monomorphism
			  for all $g\ge 1$. 
			\end{pnum}
		\end{proposition}

	\begin{proof} First of all note that $\cat{A}\neq 0$ as mentioned on page~\pageref{pageftcnonzero}.
		\begin{pnum}
			
			\item 	The case $g =0$ is clear.
			In the case $g=1$, we use \begin{align} \cat{A}(I,\mathbb{A})\cong \int_{X\in\cat{A}}\cat{A}(X,X)\cong Z(B) \ . 
			\end{align}
			(The identification of the end $\int_{X\in\cat{A}}\cat{A}(X,X)$ is with the 
			center $Z(B)$ is standard).
			Clearly, $Z(B)=0$ could only hold if $B=0$ which is forbidden since $\cat{A}\neq 0$.
			For $g \ge 2$, we have
			\begin{align} 0\neq \cat{A}(I,\mathbb{A})^{\otimes g} \subset \cat{A}(I,\mathbb{A}^{\otimes g}) \end{align} by Lemma~\ref{lemmahommap}.
			
			\item The evaluation 
			map $\cat{A}(I,\mathbb{A}^{\otimes (g-1)}) \otimes I \to \mathbb{A}^{\otimes (g-1)}$ is 
			a monomorphism 
			 by Lemma~\ref{lemmapropmor}. 
			After tensoring with $G$, we obtain, thanks to the exactness of the monoidal product, 
			a monomorphism $\cat{A}(I,\mathbb{A}^{\otimes (g-1)}) \otimes G \to \mathbb{A}^{\otimes (g-1)} \otimes G$.
			This monomorphism is non-zero because $\cat{A}(I,\mathbb{A}^{\otimes (g-1)}) \otimes G$ is isomorphic to
			$\dim \cat{A}(I,\mathbb{A}^{\otimes (g-1)})$ many copies of $G$, and
			the dimension of $ \cat{A}(I,\mathbb{A}^{\otimes (g-1)})$
			is non-zero (by statement~\ref{prophomnonzeroi})
		 and $G \neq 0$ (because $\cat{A}\neq 0$). 
			\end{pnum}
		\end{proof}

\begin{proof}[\slshape Proof of Theorem~\ref{thmmain}]
	We first prove that $d$ acts by an automorphism of $\widehat{\cat{A}}(H_g)$ of  order $|\theta|$; we  discuss the order in $\PGL(\widehat{\cat{A}}(H_g))$ afterwards.
	 After Corollary~\ref{corogamma}, it suffices to prove
	$o_\gamma = |\theta|$
	(recall that $o_\gamma$ was the order of the action on $Z\left(B; \cat{A}\left(G,G\otimes \mathbb{A}^{\otimes (g-1)}\right)\right) $ 
	with $\nu=\theta_G:G\to G$ on the $G$ in the left slot of the hom). 
	Clearly, $o_\gamma\le |\theta|$. It remains to prove $o_\gamma \ge |\theta|$. 
	To this end, we apply the left exact hom functor $\cat{A}(G,-)$ to the monomorphism from Proposition~\ref{prophomnonzero}~\ref{prophomnonzeroii} and obtain a monomorphism
	\begin{align}\label{eqntwobimodules}
		\underbrace{	\cat{A}(I,\mathbb{A}^{\otimes (g-1)}) \otimes \End_\cat{A}(G)}_{	\cat{A}(I,\mathbb{A}^{\otimes (g-1)}) \otimes B} \to  \cat{A}(G,\mathbb{A}^{\otimes (g-1)} \otimes G) \ . 
		\end{align}
		This is a monomorphism of $B$-bimodules if $B$ on the left hand side is equipped with the regular left and right action and if 
		$\cat{A}(G,\mathbb{A}^{\otimes (g-1)} \otimes G)$ is equipped with the bimodule structure explained in Lemma~\ref{lemmablockforH} (or rather its proof).
	This inclusion is equivariant with respect to the action with $\nu = \theta_G:G\to G$ on the $G$ in the left slot of the hom. 
	Taking $Z(B;-)$ of~\eqref{eqntwobimodules} preserves the monomorphism and yields a monomorphism
	\begin{align}
	\cat{A}(I,\mathbb{A}^{\otimes (g-1)}) \otimes Z(B) \to Z\left( B;  \cat{A}(G,\mathbb{A}^{\otimes (g-1)} \otimes G) \right) \ . 
	\end{align}
	We have used here the  fact  $Z(B;B)=Z(B)$. 
	
	If we denote the order of the action of $\nu$ on the left hand side by $o_\gamma'$, then $o_\gamma \ge o_\gamma'$. 
	Since  $\cat{A}(I,\mathbb{A}^{\otimes (g-1)}) \neq 0$
	by Proposition~\ref{prophomnonzero}~\ref{prophomnonzeroi}, the order $o_\gamma'$ is given by the order of the action of $\nu$ on $ Z(B)$. But $\nu$ is itself just an element of the center of $B$, so this order is just the order of the element $\nu$.
	This implies $o_\gamma'=|\nu|$, and hence $o_\gamma\ge o_\gamma' =|\nu|$.
	But $|\nu|=|\theta|$ by Proposition~\ref{proporder}.
	
	So far, we have proved that $d$ acts by an element of order $|\theta|$. But the order in $\PGL(\widehat{\cat{A}}(H_g))$ could be smaller. To exclude this possibility, assume that $\widehat{\cat{A}}(d)^\ell = \lambda \id$ for some $\ell < |\theta|$ with some $\lambda \in k^\times$ with $\lambda \neq 1$. If we denote by $\bar{\cat{A}}$ the same monoidal category as $\cat{A}$, but with braiding and ribbon twist inverted, then $\widehat{\bar{\cat{A}}\ }(d)= \widehat{\cat{A}}(d)^{-1}$  because the action of $d$ is locally just induced by the ribbon twist. 
	This implies $\widehat{\bar{\cat{A}}\ }(d)^\ell=\lambda^{-1}\id$. 
	 But now
	  $\widehat{\bar{\cat{A}} \boxtimes\cat{A}}\simeq \widehat{\bar{\cat{A}}\  } \boxtimes\widehat{\cat{A}}$ by the uniqueness of the extension from genus zero to the modular handlebody operad $\Hbdy$~\cite[Theorem 5.7]{mwansular} and
	$\left(\widehat{\bar{\cat{A}} \boxtimes\cat{A}}\right)(d)^\ell=\lambda^{-1}\id \otimes \lambda \id=\id \otimes \id$.
	This contradicts the already established statement
	that the order of $\widehat{\bar{\cat{A}} \boxtimes\cat{A}}(d)$ in $\catf{GL}(   \widehat{\bar{\cat{A}} \boxtimes\cat{A}}(H_g)   )$ must be $|\theta^{-1}\boxtimes\theta|=|\theta|$
	(note that
	$\bar{\cat{A}}\boxtimes\cat{A}$ is also a unimodular finite ribbon category, so the previously established statement applies). 
	\end{proof}

\begin{theorem}\label{thmmoddehn}
	Let $\cat{A}$ be a modular category 
	and denote by
	$\mathfrak{F}_{\! \cat{A}}$ its modular functor.
	Let $d$ be
	any Dehn twist 
	about a non-separating essential simple closed curve
	on a closed surface $\Sigma$
	with genus $g\ge 1$.
	Then $d$ acts on the space of conformal blocks $\mathfrak{F}_{\! \cat{A}}(\Sigma)$ by a linear automorphism $\mathfrak{F}_{\! \cat{A}}(d)$ whose order
	in $\PGL(\mathfrak{F}_{\! \cat{A}}(\Sigma))$
	is equal to the order $|\theta|$ of the ribbon structure of $\cat{A}$.
	In particular, $\mathfrak{F}_{\! \cat{A}}(d)$ has infinite order if and only if the ribbon structure has infinite order.
	\end{theorem}

\begin{proof}
	This is a consequence of Theorem~\ref{thmmain} and the factorization homology description of modular functors from~\cite{brochierwoike}:
	By restriction of $\mathfrak{F}_{\! \cat{A}}$ along the boundary map $\partial : \Hbdy \to \Surf$ we obtain an ansular functor that is equivalent
	to the modular extension $\widehat{\cat{A}}$ of $\cat{A}$ \cite[Theorem~6.2]{brochierwoike}. For any handlebody $H$ with $\partial H=\Sigma$, this tells us in particular
	that we have a $\Map(H)$-equivariant isomorphism 
	$\widehat{\cat{A}}(H) \ra{\cong}
	\mathfrak{F}_{\! \cat{A}}(\Sigma)$. 
	If we choose $H$ such that the Dehn twist $d$ is contained in the subgroup
	 $\Map(H)\subset \Map(\Sigma)$, the assertion follows from Theorem~\ref{thmmain}.
	\end{proof}

\begin{remark}
In the special case where $\cat{A}$ is given by finite-dimensional modules over the small quantum group (a certain ribbon factorizable Hopf algebra),
Theorem~\ref{thmmoddehn} amounts to
\cite[Proposition~5.1]{gai2} (earlier results in the same direction appear in \cite{bcgp}).
Even in this special case, the proof in \cite{gai2} is very different from ours;
most importantly, it uses three-dimensional methods.
We return to this special case in Example~\ref{examplesl2}.
\end{remark}

Since the mapping class group can be generated by Dehn twists about separating simple closed curves (see \cite[Theorem~4.11]{farbmargalit} for a textbook reference), 
Theorem~\ref{thmmoddehn} implies:

\begin{corollary}
	Let $\cat{A}$ be a modular category and $\Sigma$ a closed surface.
	Then the $\Map(\Sigma)$-action on the spaces of conformal blocks for $\Sigma$ has a system of generators that each have order $|\theta|$, where $\theta$ is the ribbon twist of $\cat{A}$. 	
	\end{corollary}

\spaceplease
\section{The action of a Dehn twist about a separating simple closed curve}
Next we turn to Dehn twists about \emph{separating} simple closed curves. More precisely, we will prove:

\begin{theorem}\label{thmmoddehn2}
	Let $\cat{A}$ be a modular category 
	and denote by
	$\mathfrak{F}_{\! \cat{A}}$ its modular functor.
	Let $d$ be
	any Dehn twist 
	about a separating essential simple closed curve
	on a closed surface $\Sigma$
	with genus $g\ge 2$ that separates the surface into pieces of genus $g'$ and $g''$.
	Then $d$ acts on the space of conformal blocks $\mathfrak{F}_{\! \cat{A}}(\Sigma)$ by a linear automorphism $\mathfrak{F}_{\! \cat{A}}(d)$ whose order
	in $\PGL(\mathfrak{F}_{\! \cat{A}}(\Sigma))$
	is equal to $\min \{  | \theta_{\mathbb{A}^{\otimes g'}}| ,  | \theta_{\mathbb{A}^{\otimes g''}}|   \}$,
	where $\mathbb{A}=\int_{X \in \cat{A}} X \otimes X^\vee\in\cat{A}$ is the canonical end of $\cat{A}$.
\end{theorem}

Once again, we reduce the statement
to a statement about the  ansular functor of a unimodular
finite ribbon category. More precisely,
Theorem~\ref{thmmoddehn2} follows from the next statement:

\begin{theorem}\label{thmsep}
	Let $\cat{A}$ be a unimodular finite ribbon category and $H_g$ a handlebody with no embedded disks and genus $g\ge 2$.
	Denote by $d \in \Map(H_g)$ a 
	 Dehn twist about a separating
	 meridian of $H_g$. 
	 We denote the genera of the handlebodies obtained by cutting along this curve by $g'$ and $g''$. 
	 Then the order of the automorphism $\widehat{\cat{A}}(d)$ in $\PGL(\widehat{\cat{A}}(H_g))$ is given by
	 \begin{align}
	 	|\widehat{\cat{A}}(d)| = \min \{  | \theta_{\mathbb{A}^{\otimes g'}}| ,  | \theta_{\mathbb{A}^{\otimes g''}}|   \} \ . 
	 	\end{align}
	\end{theorem}

It remains to prove
 Theorem~\ref{thmsep}. To this end, we first extract from \cite[Section~7.2]{cyclic}, after making simplifications similar to those in the proof of Lemma~\ref{lemmablockforH}:

\begin{lemma}\label{lemmacompsep} Let $\cat{A}$ be a unimodular finite ribbon category.
	Suppose $H_g = H_{g',1}\cup H_{g'',1}$ by cutting along a separating meridian.
	Then 
	\begin{align}
		\widehat{\cat{A}}(H_g) \cong \cat{A}(  \mathbb{A}^{\otimes g'} , \mathbb{A}^{\otimes g''}   )^* \ , 
	\end{align} 
	and the Dehn twist about the cutting curve acts by postcomposition with
	$\theta_{\mathbb{A}^{\otimes g''}}$ or, equivalently, precomposition with
	$\theta_{\mathbb{A}^{\otimes g'}}$. 
\end{lemma}

In addition, we will need the following algebraic fact:

		\begin{lemma}\label{lemmahommap2}
		Let $\cat{A}$ be a monoidal category in $\Rexf$ with exact monoidal product and simple unit.
		Then the map
		\begin{align} \End_\cat{A}(X) \otimes \cat{A}(I,Y) \to \cat{A}(X,X\otimes Y) \quad \text{for}\quad X,Y \in \cat{A}\label{eqnimagemap2}
		\end{align} induced by the monoidal product is a monomorphism.
	\end{lemma}

\begin{proof}
	As observed in the proof of Lemma~\ref{lemmahommap}, the map $\cat{A}(I,Y)\otimes X \to X \otimes Y$ is a monomorphism (this was Lemma~\ref{lemmapropmor} combined with the exactness of the monoidal product). Now we apply the left exact functor $\cat{A}(X,-)$. 
	\end{proof}

\begin{proof}[\slshape Proof of Theorem~\ref{thmsep}] 	Lemma~\ref{lemmacompsep} tells us 
	that the order of $\widehat{\cat{A}}(d)$ is the order of the action of $\theta_{\mathbb{A}^{\otimes g''}}$ or, equivalently, 
	$\theta_{\mathbb{A}^{\otimes g'}}$ on
	$\cat{A}(  \mathbb{A}^{\otimes g'} , \mathbb{A}^{\otimes g''}   )$.
	This implies
	immediately 
	\begin{align}
		|\widehat{\cat{A}}(d)| \le \min \{  | \theta_{\mathbb{A}^{\otimes g'}}| ,  | \theta_{\mathbb{A}^{\otimes g''}}|   \} \ . 
	\end{align}
	For the proof of the corresponding `$\ge$' inequality, we can assume without loss of generality $g'' \ge g'$ and write $\mathbb{A}^{\otimes g''}=\mathbb{A}^{\otimes g'} \otimes \mathbb{A}^{\otimes m}$ for some $m\ge 0$. 
	By Lemma~\ref{lemmahommap2} we have a monomorphism
		\begin{align} \iota : \End_\cat{A}(\mathbb{A}^{\otimes g'}) \otimes \cat{A}(I,\mathbb{A}^{\otimes m}) \to \cat{A}(\mathbb{A}^{\otimes g'} , \mathbb{A}^{\otimes g'} \otimes \mathbb{A}^{\otimes m} ) \ . 
			\label{eqnimagemap3}\end{align}
			This monomorphism is non-zero because $\cat{A}(I,\mathbb{A}^{\otimes m}) \neq 0$ 
			by Proposition~\ref{prophomnonzero}~\ref{prophomnonzeroi} and 
			$\End_\cat{A}(\mathbb{A}^{\otimes g'}) \neq 0$ 
			(because otherwise $\mathbb{A}^{\otimes g'} =0$ which would contradict Proposition~\ref{prophomnonzero}~\ref{prophomnonzeroi}).
			Moreover, \eqref{eqnimagemap3} is $\mathbb{Z}$-equivariant if we equip the right hand side with the Dehn twist action
			(the action of $\theta_{\mathbb{A}^{\otimes g''}}=\theta_{\mathbb{A}^{\otimes g'} \otimes \mathbb{A}^{\otimes m}}$ or, equivalently, 
			$\theta_{\mathbb{A}^{\otimes g'}}$) and the left hand side with the action of $\theta_{\mathbb{A}^{\otimes g'}}$ on $\End_\cat{A}(\mathbb{A}^{\otimes g'})$.
			This follows from the naturality of $\theta$ and $\theta_I=\id_I$. 
			Now let $\ell < \min \{  | \theta_{\mathbb{A}^{\otimes g'}}| ,  | \theta_{\mathbb{A}^{\otimes g''}}|   \}$ and pick $0\neq v \in \cat{A}(I,\mathbb{A}^{\otimes m})$. 
			The $\ell$-th power of $d$ sends $\iota (\id_{\mathbb{A}^{\otimes g'}}     \otimes v)$ to 
			$\iota (\theta_{\mathbb{A}^{\otimes g'}}^\ell \otimes v)$.
			Moreover, $\theta_{\mathbb{A}^{\otimes g'}}^\ell \neq \id_{\mathbb{A}^{\otimes g'}}$ by
			assumption. 
			By injectivity of $\iota$
			we arrive at
			$ \iota (\theta_{\mathbb{A}^{\otimes g'}}^\ell \otimes v) \neq \iota (\id_{\mathbb{A}^{\otimes g'}}     \otimes v)$, thereby proving $|\widehat{\cat{A}}(d)| \ge \min \{  | \theta_{\mathbb{A}^{\otimes g'}}| ,  | \theta_{\mathbb{A}^{\otimes g''}}|   \}$. 
			
			This proves
			$|\widehat{\cat{A}}(d)| = \min \{  | \theta_{\mathbb{A}^{\otimes g'}}| ,  | \theta_{\mathbb{A}^{\otimes g''}}|   \}$ in $\catf{GL}(\widehat{\cat{A}}    (H_g))$. In order to determine the order in $\PGL(\widehat{\cat{A}}    (H_g))$,
			let us assume $\widehat{\cat{A}}(d)^\ell=\lambda \id$ for some $1 \neq \lambda \in k^\times$ and some $1<\ell < \min \{  | \theta_{\mathbb{A}^{\otimes g'}}| ,  | \theta_{\mathbb{A}^{\otimes g''}}|   \}$. We will show that this leads to a contradiction.
			With the exact same argument as in the proof of Theorem~\ref{thmmain},
			we conclude that
			$\left(\widehat{\bar{\cat{A}} \boxtimes\cat{A}}\right)(d)^\ell$ is the identity, which after applying the already established statement for
			the order in 
			the general linear group to $\bar{\cat{A}} \boxtimes \cat{A}$ tells us \begin{align}
			\label{eqnellge}	\ell \ge \min \{  | \theta_{\mathbb{A}^{\otimes g'}_{   \bar{\cat{A}} \boxtimes \cat{A}        }}| ,  | \theta_{\mathbb{A}^{\otimes g''}_{\bar{\cat{A}} \boxtimes \cat{A}}}|   \}\ , \end{align}
			where $\mathbb{A}^{\otimes p}_{\bar{\cat{A}} \boxtimes \cat{A}}$ for $p\ge 0$ is the end in $\bar{\cat{A}} \boxtimes \cat{A}$, namely $\mathbb{A}^{\otimes p} \boxtimes \mathbb{A}^{\otimes p}$ with ribbon twist $\theta_{\mathbb{A}^{\otimes p}}^{-1} \boxtimes \theta_{\mathbb{A}^{\otimes p}}$. 
			If we knew
			\begin{align}
			\label{eqnminima}	\min \left\{  \left| \theta_{\mathbb{A}^{\otimes g'}_{   \bar{\cat{A}} \boxtimes \cat{A}        }}\right| ,  \left| \theta_{\mathbb{A}^{\otimes g''}_{\bar{\cat{A}} \boxtimes \cat{A}}}\right|   \right\}
				= \min \{  | \theta_{\mathbb{A}^{\otimes g'}}| ,  | \theta_{\mathbb{A}^{\otimes g''}}|   \} \ , 
				\end{align}
			we would be done because we would obtain a contradiction to the definition of $\ell$.
			The equality~\eqref{eqnminima} will follow from the equality
			\begin{align} \left|\theta_{\mathbb{A}^{\otimes p}}^{-1} \boxtimes \theta_{\mathbb{A}^{\otimes p}}\right| = |\theta_{\mathbb{A}^{\otimes p}}|   \label{twistAAeq} \  \end{align}
			that we will prove now:
			First note that only the $\ge$-inequality is a priori
			not obvious; it would fail e.g.\
			if $\theta_{\mathbb{A}^{\otimes p}}=\xi \id$ for some $\xi \neq 1$. 
			To see the $\ge $-inequality, we use the monomorphism $u:I \to \mathbb{A}$ that selects the unit of the canonical end algebra 
			(the map induced by the  coevaluations $I\to X\otimes X^\vee$ for all $X\in \cat{A}$; it is a monomorphism because the projection $\mathbb{A} \to I \otimes I^\vee \cong I$ splits it). Since $\otimes$ is exact,
			the  map $I\ra{u^{\otimes p}} \mathbb{A}^{\otimes p}$ is also a monomorphism, and so is the map
			\begin{align}
				I \boxtimes \mathbb{A}^{\otimes p} \ra{u^{\otimes p}    \boxtimes \id_{\mathbb{A}^{\otimes p} }      }  \mathbb{A}^{\otimes p} \boxtimes \mathbb{A}^{\otimes p} \ . 
				\end{align}  This inclusion intertwines the twist by naturality and proves 
			the $\ge$-inequality in~\eqref{twistAAeq} and therefore~\eqref{eqnminima}. This concludes the proof.
	\end{proof}

\begin{corollary}\label{corbalAdehntwist}
	Let $\cat{A}$ be a unimodular
	finite ribbon category, 
	and let $\gamma$ be the separating simple closed curve on $\partial H_2$ that cuts $H_2$ into two genus one pieces.
	Then the $p$-th power of Dehn twist about $\gamma$ acts trivially on 	$\widehat{\cat{A}}(H_2)$ if and only if $\theta_\mathbb{A}^p=\id_\mathbb{A}$ for the end $\mathbb{A}=\int_{X\in\cat{A}} X\otimes X^\vee $. 
\end{corollary}

\begin{remark}\label{rembeyondrigid}
Theorem~\ref{thmmain} and~\ref{thmmoddehn} are formulated for ansular and modular functors coming from unimodular finite ribbon categories and modular categories, respectively.
One could ask whether the statement still holds for general ribbon Grothendieck-Verdier categories (whose Grothendieck-Verdier duality is not necessarily the rigid duality). 
Generally, it turns out to be false because,
	for $g\ge 2$, 
	 we can build ansular and modular functors with non-trivial balancing whose space of conformal blocks in genus $g$ is zero. This is a consequence of \cite[Example~3.4]{mwcenter}.
\end{remark}

\section{Applications, special cases and examples}

\subsection{Unitary structures}
When investigating modular functors from complex
modular fusion categories, i.e.\ in the semisimple case,
a large part of the literature focuses on the unitary case, where the mapping class group acts through unitary operators, see e.g.~\cite[Section~2.4.1]{rw}. 
The order of the Dehn twists has implications for the existence of unitary structures on the spaces of conformal blocks.
This insight seems to be standard and is reviewed in~\cite[Section~4.1]{funartqft}. 
Let us give here the full argument for completeness. Afterwards, we use this to derive
 consequences from 
Theorem~\ref{thmmoddehn}.

\begin{lemma}
	Let $\cat{A}$ be a modular category.
	If for some closed surface $\Sigma$ of genus $g\ge 3$ the associated mapping class group representation 
	$\Map(\Sigma) \to \PGL(  \mathfrak{F}_{\! \cat{A}}(\Sigma)   )$
	factors through a group morphism $\Map(\Sigma)\to G$ for a compact Lie group $G$,
	then any Dehn twist acts by a finite order element under the representation.
	\end{lemma}

\begin{proof}Let $\Sigma$ be a closed surface of genus $g\ge 3$. 
	Suppose the representation 
	factors through a compact Lie group $G$, i.e.\ $\varrho : \Map(\Sigma) \to \PGL(  \mathfrak{F}_{\! \cat{A}}(\Sigma)   )$ can be written as the composition
	\begin{align}
		\Map(\Sigma) \ra{\varphi} G \ra{\psi} \PGL(  \mathfrak{F}_{\! \cat{A}}(\Sigma)   ) \ . 
	\end{align}
	The map $\varphi$ sends every Dehn twist  $d$
	to an element of finite order $\ell(d)$ in $G$ by \cite[Corollary~2.6]{AramayonaSouto}. 
	But then $\varrho(d)^{\ell(d)}$ is the identity.
	\end{proof}

We can therefore draw the following conclusion from Theorem~\ref{thmmoddehn}:

\begin{corollary}\label{corlie}
	Let $\cat{A}$ be a  modular category 
	 with $|\theta|=\infty$.
	Then the associated mapping class group representation 
	$\Map(\Sigma) \to \PGL(  \mathfrak{F}_{\! \cat{A}}(\Sigma)   )$
	for a surface of genus $g\ge 3$ does not factor through a group morphism $\Map(\Sigma)\to G$ for any compact Lie group $G$.
	In particular, the representation cannot be made unitary if $\cat{A}$ lives over the complex numbers.
\end{corollary}

\subsection{Extending faithfulness beyond single Dehn twists}
The faithful translation of the ribbon twist
 to the operator assigned to Dehn twists 
about non-separating simple closed curves is still 
far away from actual faithfulness statements for the entire representation.
We can, however, extend the faithfulness from one single Dehn twist to a slightly larger subgroup.
\begin{figure}[h]
\begin{center}
\begin{overpic}[scale=0.6
	,tics=10]
	{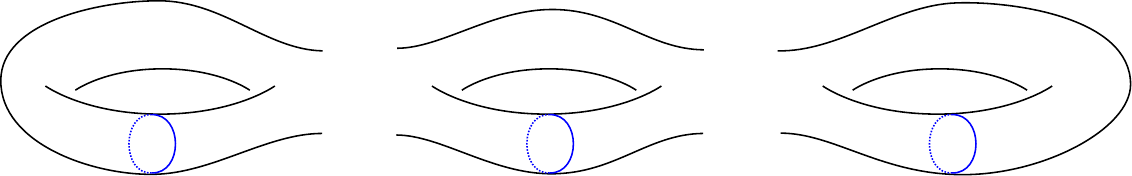}
	\put(30,7){$\dots$} 
	\put(63.5,7){$\dots$} 
\end{overpic}
\end{center}
\caption{The Dehn twists generating the subgroup $\Z^g \subset \Map(\Sigma_g)$ in Proposition~\ref{propzg}}
\label{figzg}
\end{figure}
\begin{proposition}\label{propzg}
	Let $\cat{A}$ be a modular category and $\mathbb{Z}^g \subset \Map(\Sigma_g)$ for $g\ge 1$ the free abelian subgroup generated by the Dehn twists shown in
	Figure~\ref{figzg}.
	Then the kernel of the $\mathbb{Z}^g$ action on $\mathfrak{F}_\cat{A}(\Sigma_g)$ is $N \mathbb{Z}^g$ with $N:=|\theta|\in\mathbb{N}\cup \{\infty\}$.
	In particular, the $\mathbb{Z}^g$-action is faithful if $|\theta|=\infty$.
\end{proposition}

\begin{proof}
	The definition of $\mathbb{Z}^g$ makes use of a handlebody with boundary $\Sigma_g$ by means of which we can identify the dual of 
	 $\mathfrak{F}_\cat{A}(\Sigma_g)$ with $\cat{A}(I,\mathbb{A}^{\otimes g})$, 
	 see Lemma~\ref{lemmacompsep}. This vector space contains the non-trivial subspace $\cat{A}(I,\mathbb{A})^{\otimes g}$ by Lemma~\ref{lemmahommap}.
	The inclusion  $\cat{A}(I,\mathbb{A})^{\otimes g} \subset \cat{A}(I,\mathbb{A}^{\otimes g})$ is $\mathbb{Z}^g$-equivariant if $\cat{A}(I,\mathbb{A})^{\otimes g}$ is seen as the $g$-th tensor product of the $\mathbb{Z}$-representation $\cat{A}(I,\mathbb{A})$. Here $\cat{A}(I,\mathbb{A})$ is the dual space of conformal blocks for the torus, and the $\mathbb{Z}$-action is the action of the $\mathbb{Z}$-factor
	in $\Map(H_1)\cong \mathbb{Z}\times\mathbb{Z}_2$.
	 Hence, the $\mathbb{Z}$-action on $\cat{A}(I,\mathbb{A})$
	  has order $N$ by Theorem~\ref{thmmoddehn}, even in the projective linear group. Therefore, it suffices to show
	that the kernel of the $\mathbb{Z}^g$-representation on $V^{\otimes g}$ with $V:=\cat{A}(I,\mathbb{A})$ is $N\mathbb{Z}^g$.  
	To this end, denote by $L \in \Aut(V)$ the image of $1\in\mathbb{Z}$ under the representation.
	Now $(n_1,\dots,n_g)\in\mathbb{Z}^g$ acts by
	$L^{n_1}\otimes\dots\otimes L^{n_g}$. 
	For $1\le i\le g$, we write
	$L^{n_1}\otimes\dots\otimes L^{n_g}=R\otimes L^{n_i}\otimes T$ with
	$R:=L^{n_1} \otimes \dots \otimes L^{n_{i-1}}$ and $T:=L^{n_{i+1}} \otimes \dots \otimes L^{n_g}$. 
	If $L^{n_1}\otimes\dots\otimes L^{n_g}=\id_{V^\otimes g}$, then for $Q \in \End(V)$
	\begin{align}
		(	\id_{V^{\otimes (i-1)}} \otimes Q \otimes \id_{V^{\otimes (g-i)}}) \circ (L^{n_1}\otimes\dots\otimes L^{n_g}) = (L^{n_1}\otimes\dots\otimes L^{n_g})\circ   (	\id_{V^{\otimes (i-1)}} \otimes Q \otimes \id_{V^{\otimes (g-i)}}) \ ,   \end{align}
	and hence
	\begin{align} R \otimes QL^{n_i} \otimes T=R \otimes L^{n_i} Q \otimes T \ . 
	\end{align}Since $R$ and $T$ are invertible, this implies
	\begin{align} \id_{V^{\otimes (i-1)}} \otimes QL^{n_i} \otimes \id_{V^{\otimes (g-i)}} = \id_{V^{\otimes (i-1)}} \otimes L^{n_i} Q \otimes \id_{V^{\otimes (g-i)}} \ . 
	\end{align}
	Since $\End(V) \mapsto \End(V^{\otimes g}), W \mapsto \id_{V^{\otimes (i-1)}} \otimes W \otimes \id_{V^{\otimes (g-i)}}$ is injective, we arrive at
	$QL^{n_i}=L^{n_i}Q$. 
	This proves that $L^{n_i}$ is central in $\End(V)$, i.e.\ $L^{n_i}=\lambda_i  \id_V$ with some $\lambda_i \in k^\times$. Since the order of $L$ in the projective linear group of $V$ is $N$, we obtain $n_i \in N\mathbb{Z}$. 
\end{proof}

By \cite[Theorem~8.2]{henselprimer} the maximal free abelian subgroup of $\Map(H_g)$ for $g\ge 2$ is $\mathbb{Z}^{3g-3}$, and the subgroup $\mathbb{Z}^g$ featuring in Proposition~\ref{propzg} is a proper
subgroup of this maximal free abelian subgroup.
We do not know whether the full maximal free abelian subgroup $\mathbb{Z}^{3g-3}$ also acts faithfully.

\subsection{Visibility of the Johnson kernel}
The \emph{Johnson kernel}~\cite{johnson} is the 
subgroup of the mapping class group generated by
Dehn twists about 
separating simple closed curves.
The following result tells us when the modular functor of a modular category will see the Johnson kernel:
\begin{proposition}\label{propjohnson}
	The modular functor of 
	a modular category $\cat{A}$ annihilates
	the Johnson kernels of all closed surfaces
	 if and only if the twist of the canonical end $\mathbb{A}=\int_{X\in\cat{A}}X\otimes X^\vee$ is trivial and if $c_{\mathbb{A},\mathbb{A}}^2=\id_{\mathbb{A}\otimes\mathbb{A}}$.
	\end{proposition}

\begin{proof}
	By Theorem~\ref{thmmoddehn2} the Johnson kernels are annihilated if and only if $\theta_{\mathbb{A}^{\otimes n}}=\id_{\mathbb{A}^{\otimes n}}$ for all $n\ge 1$.
	Thanks to~\eqref{eqnbalancing}, we obtain \begin{align}\theta_{\mathbb{A}^{\otimes n}} = c_{\mathbb{A}^{\otimes (n-1)},\mathbb{A}}c_{\mathbb{A},\mathbb{A}^{\otimes (n-1)}}\circ (\theta_\mathbb{A} \otimes \theta_{\mathbb{A}^{\otimes (n-1)}})\ . \label{eqnfromeqnbalancing}
	\end{align}
	We can  deduce that
	\begin{align}
\label{eqndedeqnfromeqnbalancing}	\theta_{\mathbb{A}^{\otimes n}}=\id_{\mathbb{A}^{\otimes n}} \ \text{for all}\ n\ge 1 \qquad \Longleftrightarrow\qquad  \theta_{\mathbb{A}}=\id_{\mathbb{A}}\ \text{and}\ c_{\mathbb{A},\mathbb{A}}c_{\mathbb{A},\mathbb{A}}=\id_{\mathbb{A}\otimes\mathbb{A}} \ . 
	\end{align}
	Indeed, for the $(\Rightarrow)$-direction, we evaluate in $n=1$ and $n=2$.
	For $(\Leftarrow)$, note that the representation of the
	 braid group on $n$ strands on $\mathbb{A}^{\otimes n}$ that exists by virtue of $\cat{A}$ being braided descends to the symmetric group thanks to $c_{\mathbb{A},\mathbb{A}}c_{\mathbb{A},\mathbb{A}}=\id_{\mathbb{A}\otimes\mathbb{A}}$, thereby annihilating the pure braid group. This tells us  $c_{\mathbb{A}^{\otimes (n-1)},\mathbb{A}}c_{\mathbb{A},\mathbb{A}^{\otimes (n-1)}}=\id_{\mathbb{A}^{\otimes n}}$.
	With~\eqref{eqnfromeqnbalancing}, we obtain $\theta_{\mathbb{A}^{\otimes n}}=\id_{\mathbb{A}^{\otimes n}}$ for all $n\ge 0$. 
	This proves~\eqref{eqndedeqnfromeqnbalancing} which, in turn, implies the assertion.
	\end{proof}

\begin{example}\label{examplesl2}
	In \cite[Section~5.2]{gai2} the order of Dehn twists are proven to be infinite in the representations associated 
	to the modular category given by finite-dimensional modules over the small quantum group $\bar U_q (\mathfrak{sl}_2)$, with $q = \exp (2\pi \text{i}/r)$ for an odd integer $r\ge 3$. For Dehn twists about non-separating curves, this is by Theorem~\ref{thmmoddehn} the `generic' behavior that needs the ribbon twist to have infinite order, but otherwise has relatively little to do with the specific example of the small quantum group.
	In the separating case, however, the result tells us in combination with Corollary~\ref{corbalAdehntwist} that the ribbon twist of $\mathbb{A}$
	(in this case, this is $\bar U_q (\mathfrak{sl}_2)_\text{ad}$, i.e.\
	$\bar U_q (\mathfrak{sl}_2)$ with its adjoint action) has a ribbon twist of infinite order.
	This also tells us that $\bar U_q (\mathfrak{sl}_2)_\text{ad}$ when seen as symmetric Frobenius algebra in $\bar U_q (\mathfrak{sl}_2)$-modules (see \cite[Lemma~3.5]{dmno} and \cite[Theorem~6.1~(4)]{shimizuunimodular})
	cannot be braided commutative with respect to the braiding of $\bar U_q (\mathfrak{sl}_2)$-modules
	(but it is braided commutative with respect to the braiding of modules over the quantum double of $\bar U_q (\mathfrak{sl}_2)$, see 
	\cite[Theorem~6.1~(1)]{shimizuunimodular}).
	This is because otherwise the ribbon twist of 
	$\bar U_q (\mathfrak{sl}_2)_\text{ad}$
	would be trivial by \cite[Proposition~2.25]{correspondences}.
	\end{example}
	
\subsection{Visibility of	 the Torelli group}
The \emph{Torelli group}
\begin{align}
	\mathcal{I}(\Sigma) := \ker \left(   \Map(\Sigma) \to \catf{Aut}(H_1(\Sigma;\mathbb{Z}))  \right)
	\end{align}
	is the subgroup of the mapping class group formed by those mapping classes acting trivially on the first homology of the surface. 
	We will need a few key facts about the Torelli group for which we
	refer to~\cite{hatchermargalit} (see also the references therein explaining how these results can be deduced from work of Birman and Powell):
	The Torelli group
	is trivial for $g=1$, generated by Dehn twists about separating simple closed curves for $g=2$ (making it identical to the Johnson kernel), and generated by bounding pair maps for $g\ge 3$ (in fact, one may restrict to those bounding pairs that cut off a genus one surface).
	The Torelli group fits into the short exact sequence
	\begin{align}
		1\to \mathcal{I}(\Sigma)\to\Map(\Sigma)\to\catf{Sp}_{2g}(\mathbb{Z})\to 1 \ ,  \end{align}
		i.e.\ $\Map(\Sigma)/\mathcal{I}(\Sigma)$ is the symplectic group $\catf{Sp}_{2g}(\mathbb{Z})$.
	
\begin{lemma}\label{lemmasepboundingpair}
	Let $\cat{A}$ be a unimodular ribbon category in $\Rexf$ such that the canonical end $\mathbb{A}$ lies in the Müger center ($c_{X,\mathbb{A}}c_{\mathbb{A},X} = \id_{X\otimes\mathbb{A}}$ for all $X\in\cat{A}$) and has trivial ribbon twist.
	For any handlebody $H$ without embedded disks, \begin{pnum}
		\item any Dehn twist in $\Map(H)$ \label{lemmasepboundingpairi}
	about a separating  meridian \item  and any \label{lemmasepboundingpairii}
	bounding pair maps of two meridians
	cutting off a genus one surface
	\end{pnum}
	acts trivially on the space of conformal blocks $\widehat{\cat{A}}(H)$.  
\end{lemma}

\begin{proof}
	By assumption $\theta_\mathbb{A}=\id_\mathbb{A}$ and $c_{\mathbb{A},X}c_{X,\mathbb{A}}=\id_{X\otimes\mathbb{A}}$ for all $X\in\cat{A}$.  
	This implies
	\begin{align} \theta_{\mathbb{A} \otimes \mathbb{A}^{\otimes \ell}} = c_{ \mathbb{A}^{\otimes \ell} , \mathbb{A}  }c_{  \mathbb{A},\mathbb{A}^{\otimes \ell} }
		(   \theta_\mathbb{A} \otimes \theta_{\mathbb{A}^{\otimes \ell}}  ) = \id_{\mathbb{A}} \otimes \theta_{\mathbb{A}^{\otimes \ell}} \quad \text{for all}\quad \ell\ge 0
	\end{align}
	and therefore inductively
	$\theta_{\mathbb{A}^{\otimes \ell}}=\id_{\mathbb{A}^{\otimes \ell}}$ for all $\ell\ge 0$. Now Theorem~\ref{thmsep} implies that any Dehn twist about a separating meridian acts trivially.
	This proves~\ref{lemmasepboundingpairi}.
	
	For the proof of~\ref{lemmasepboundingpairii}, note that by the construction of the ansular functor associated to $\cat{A}$~\cite{cyclic,mwansular} the torus with two boundary components (one thought of as incoming, one as outgoing) is sent to the (dual of the) vector space $\cat{A}(X\otimes \mathbb{A}, Y)$ with the action of the bounding pair maps being the `difference' of the ribbon twist on  $X$ and the ribbon twist on  $Y$. 
	In other words, the action is $f\mapsto \theta_Y \circ f \circ (\id_\mathbb{A}\otimes \theta_X^{-1}) $ for any morphism $f:X \otimes \mathbb{A}\to Y $. By the excision property for spaces of conformal blocks it suffices to prove that this action is trivial. 
	Indeed, we obtain
	\begin{align}
	\theta_Y \circ f \circ (\id_\mathbb{A}\otimes \theta_X^{-1}) = f \circ \theta_{\mathbb{A}\otimes X} \circ (\id_\mathbb{A}\otimes \theta_X^{-1}) = f \circ c_{\mathbb{A},X}c_{X,\mathbb{A}} \circ (\theta_{\mathbb{A}} \otimes \id_X)
		\end{align} By assumption on $\mathbb{A}$ the right hand side is just equal to $f$. This  finishes the proof of~\ref{lemmasepboundingpairii}.
\end{proof}

\begin{proposition}\label{proptorelli}
	Let $\cat{A}$ be a modular category.
	The representations $\mathfrak{F}_\cat{A}(\Sigma)$ for all closed surfaces $\Sigma$ 
	annihilate the Torelli group
	(this means that they descend to representations of the symplectic group) if and only if
the	canonical end $\mathbb{A}$ lies in the Müger center, i.e.\ $\mathbb{A}\cong I^{\oplus n}$ for some $n\ge 1$.
	\end{proposition}
	
	\begin{proof} Suppose $\mathbb{A}\cong I^{\oplus n}$ for some $n\ge 1$. Let us now prove that the Torelli groups of closed surfaces act trivially.
		Since $\theta_I=\id_I$, the hypotheses of Lemma~\ref{lemmasepboundingpair} are satisfied.
		For $g=1$, there is nothing to show. For $g=2$, we can use Lemma~\ref{lemmasepboundingpair}~\ref{lemmasepboundingpairi} because the separating Dehn twist generate the Torelli group. For $g=3$, we can use Lemma~\ref{lemmasepboundingpair}~\ref{lemmasepboundingpairii} because the Torelli group is generated by bounding pair maps of two meridians cutting off a genus one subsurface.
		(All of this uses of course the above-mentioned facts about the generation of the Torelli groups from~\cite{hatchermargalit}.)
		
		Conversely, let us assume that the representations of $\Map(\Sigma)$
		built from $\cat{A}$ annihilate the Torelli group 
		 for all closed surfaces $\Sigma$. By Corollary~\ref{corbalAdehntwist} the ribbon twist of $\mathbb{A}$ is trivial. 
		 The space of conformal blocks for a genus two surface with two boundary components both labeled by $X$ (one incoming, one outgoing) is $\cat{A}(X \otimes \mathbb{A}^{\otimes 2} , X)^*$. The bounding pair map for the genus one subsurface shown in Figure~\ref{figDT}
		 acts via the ribbon twist on $X$ and the inverse ribbon twist on $X\otimes \mathbb{A}$
		  as in the proof of Lemma~\ref{lemmasepboundingpair}.
		  
		  \begin{figure}[h]
		  	\begin{center}
		  		\begin{overpic}[scale=0.8
		  			,tics=10]
		  			{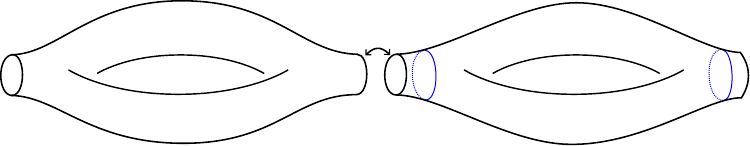}
		  			\put(-3.5,8){$X$}
		  			\put(100.5,8){$X$}
		  		\end{overpic}
		  	\end{center}
		  	\caption{The Dehn twists producing the bounding pair map for a genus two surface 
		  		with two boundary components (one incoming, one outgoing) labeled by $X$. The simple closed curves for the Dehn twists are indicated in blue. The arrow indicates a gluing. }
		  	\label{figDT}
		  \end{figure}
		  
		 Now consider the case where $X$ is a projective generator $G$ of $\cat{A}$.
		 With the projective Calabi-Yau structure coming from the modified trace~\cite[Corollary~6.3]{mtrace}, 
		 $\cat{A}(G \otimes \mathbb{A}^{\otimes 2} , G)^*\cong \cat{A}(G,G\otimes\mathbb{A}^{\otimes 2})$.
		 Thanks to unimodularity $\mathbb{A}\cong \mathbb{A}^\vee$, so that the space of conformal blocks is $\End_\cat{A}(G\otimes \mathbb{A})$ with the bounding pair map acting via the automorphism sending $f: G \otimes \mathbb{A}\to G\otimes\mathbb{A}$ to
		$(\theta_G \otimes \id_\mathbb{A} ) \circ f \circ \theta_{G \otimes \mathbb{A}}^{-1}$. 
		Thanks to $\theta_\mathbb{A}=\id_\mathbb{A}$, the action is
	$ f \mapsto f \circ (c_{\mathbb{A},G}c_{G,\mathbb{A}})^{-1} $. 
	By excision the space of conformal blocks for the closed surface of genus three is the subspace
	of $\End_\cat{A}(G\otimes \mathbb{A})$ given by those maps $f: G\otimes \mathbb{A}\to G \otimes \mathbb{A}$ that are maps of $\End_\cat{A}(G)$-modules. Clearly, $\id_{G\otimes \mathbb{A}}$ is therefore a vector in the space of conformal blocks for $\Sigma_3$. If the bounding pair maps described above, seen as elements of the Torelli group of $\Sigma_3$, acts trivially, then it must fix the vector $\id_{G\otimes \mathbb{A}}$
	which implies that $c_{\mathbb{A},G}c_{G,\mathbb{A}}$ is the identity.
	But if $\mathbb{A}$ trivially double braids with a projective generator, it trivially double braids with $G^{\oplus m}$ for any $m\ge 0$. Since any $X\in \cat{A}$ admits an epimorphism $G^{\oplus m}\to X$ for some $m$, we conclude that $\mathbb{A}$ trivially double braids with any $X\in \cat{A}$.
		\end{proof}

		\spaceplease
		\begin{example}[Visibility of the Torelli groups as a non-commutative phenomenon --- generalization of Fjelstad-Fuchs~\cite{ff}]
			Let $H$ be a ribbon factorizable Hopf algebra and $\cat{A}$ the category of finite-dimensional $H$-modules.
			By \cite[Theorem~7.4.13]{kl} $\mathbb{A}$ is $H$ with the adjoint representation. 
			The object $\mathbb{A}$ is in the Müger center if and only if $H$ is commutative.
			Indeed, if
			 $H$ is commutative, then $\mathbb{A}$ is  isomorphic to $\dim H$ many copies of the unit (which is the ground field $k$ with the counit action). 
			 Conversely, if $H_\text{ad}$ is in the Müger center, i.e.\ $H_\text{ad}\cong I^{\oplus m}$ for some $m\ge 1$, then $H_\text{ad}$ is the $H$-module with $h.x = \varepsilon(h)x$ for all $h,x\in H$. 
			 This tells us $\Hom_H(k,H_\text{ad})=H$. But
			 $\Hom_H(k,H_\text{ad})=\cat{A}(I,\mathbb{A})\cong\int_{X\in\cat{A}}\cat{A}(X,X)\cong Z(H)$.
			 Therefore, $H=Z(H)$, making $H$ commutative.
			We now conclude from Proposition~\ref{proptorelli} that the mapping class group representations built from a  ribbon factorizable Hopf algebra annihilate the Torelli groups of all closed surfaces if and only if $H$ is commutative.
			If we choose for $H$ the Drinfeld double of a finite abelian group $G$, then the associated modular functor is the Dijkgraaf-Witten modular functor for $G$. Then the statement reduces to the result by Fjelstad and Fuchs~\cite{ff} that the Dijkgraaf-Witten type modular functor sees the Torelli group if and only if $G$ is not abelian.
			\end{example}

\begin{example}
A modular fusion category is called \emph{pointed} if any simple object $X$ is invertible, i.e.\ $X \otimes X^\vee \cong I$.
	For a  modular fusion category $\cat{A}$
with simple objects $X_0,\dots,X_n$, we have $\mathbb{A} \cong \bigoplus_{i=0}^n X_i \otimes X_i^\vee $.
If $\cat{A}$ is pointed, each $X_i$ is invertible, i.e.\ $X_i^\vee \otimes X_i \cong I$.
As a result $\mathbb{A}\cong I^{\oplus (n+1)}$ lies in the  Müger center of $\cat{A}$. 
We conclude that the
Torelli groups of closed surfaces act trivially on the space of conformal blocks of a pointed modular fusion category.
\end{example}

	\small	
\newcommand{\etalchar}[1]{$^{#1}$}

\vspace*{0.3cm} \noindent \textsc{Perimeter Institute,  N2L 2Y5 Waterloo, Canada} \\[2ex] \noindent \textsc{Université Bourgogne Europe, CNRS, IMB UMR 5584, F-21000 Dijon, France}

\end{document}